\documentclass[11pt]{amsart}

\usepackage{amsthm, mathrsfs,amssymb,amsmath}
\usepackage{enumerate}
\usepackage{kbordermatrix}
\newcommand\vneq{\mathrel{\ooalign{$=$\cr\hidewidth$|$\hidewidth\cr}}}

\makeatletter
\@namedef{subjclassname@2010}{%
\textup{2010} Mathematics Subject Classification}
\makeatother

\newcommand{\mychi}{\raisebox{0pt}[1ex][1ex]{$\chi$}}
\newcommand{\myxi}{\raisebox{0pt}[1ex][1ex]{$\xi$}}
\newcommand{\myg}{\raisebox{0pt}[0.5ex][0.5ex]{$\gamma$}}

\allowdisplaybreaks
\newcommand{\eqspecialnum}{
  \renewcommand{\theequation}{\arabic{equation}*}}

\newtheorem{thm}{Theorem}[section]

\newtheorem{prop}[thm]{Proposition}
\newtheorem{cor}[thm]{Corollary}

\theoremstyle{definition}
\newtheorem{defn}[thm]{Definition} 

\newtheorem{exmp}{Example}[section]

\theoremstyle{remark}
\newtheorem{rem}{Remark}

\newcommand{\spt}{\textnormal{spt}}

\title[Hypergroup Deformations of Semigroups]{Hypergroup Deformations of Semigroups}

\author{Vishvesh Kumar}
\address{Vishvesh Kumar \endgraf Department of Mathematics \endgraf Indian Institute of Technology Delhi \endgraf New Delhi - 110 016, India.} 
\email{vishveshmishra@gmail.com}
\author{Kenneth A. Ross} 
\address{ Kenneth A. Ross, \endgraf Prof Emeritus  \endgraf Department of Mathematics \endgraf University of Oregon \endgraf  Eugene, OR 97403, USA.} 
\email{kenross.math@gmail.com}
\author{Ajit Iqbal Singh}
\address{Ajit Iqbal Singh \endgraf INSA Emeritus Scientist\endgraf The Indian National Science Academy \endgraf New Delhi - 110002, India.} 
\email{ajitis@gmail.com}

\begin{document}

\begin{abstract} We view the well-known example of the dual of a countable compact hypergroup, motivated by the orbit space of p-adic integers by Dunkl and Ramirez (1975), as hypergroup deformation of the max semigroup structure on the linearly ordered set $\mathbb{Z}_+$ of the non-negative integers along the diagonal. This works as motivation for us to study hypergroups or semi convolution spaces arising from ``max" semigroups or general commutative semigroups via hypergroup deformation on idempotents. 

\end{abstract}
\maketitle

\section{Introduction}
We introduce and study hypergroups, same as convolution spaces, in short, convos, or, semi convolution spaces, in short, semiconvos \cite{Jewett} arising from general commutative semigroups via deformation of the part of diagonal consisting of the idempotents. The genesis was the well-known example related to the orbit space of p-adic integers by Dunkl and Ramirez \cite{Ramirez} viewed with this perspective on one hand and a good account of the structure of measure algebras of certain linearly ordered semigroups with order topology in \cite{Zuck} and \cite{Ross} on the other hand (see also \cite{Hofmann}). There is a substantial development of hypergroup deformations of groups and their applications. To get an idea one can see \cite{Xu} and \cite{Kawakami}. The present work can be seen as a complement of that because a group has no idempotents other than the identity.

The next section gives basics of semigroups and hypergroups in the form that we need with a little touch of novelty at a few places. We begin Section 3 with an attempt to make a ``max" semigroup $(S, <, \cdot)$ with the discrete topology into a hermitian discrete hypergroup by deforming the product on the diagonal. Amongst other things, we arrive at the result that this can be done if and only if either $S$ is finite or $S$ is isomorphic to $(\mathbb{Z}_+,<,\text{max}).$ We determine its  dual and show that it becomes a countable compact hermitian  hypergroup with respect to pointwise multiplication. In Section 4, we prove our main theorem on semiconvo or hypergroup deformations of idempotents in commutative semigroups.

Let $\mathbb{Z}_+= \mathbb{N} \cup \{0\}$ and $\times,$ the usual multiplication. For a subset $T$ of $S,$ $\mychi_T$ denotes the characteristic function of $T$ defined on $S.$ For notational convenience, we take empty sums to be zero. 
\section{Basics of semigroups and hypergroups} \label{Basic}
For basics of semigroups and hypergroups one can refer to standard  books, monographs and research papers. For instance, one can see (\cite{Hewitt}, \cite{DR}, \cite{Hugo}, \cite{Wolf} \cite{Zuck}, \cite{Ross}, \cite{Hofmann}) for semigroups and (\cite{Dunkl}, \cite{Jewett}, \cite{Spector}, \cite{Ramirez},  \cite{Bloom}, \cite{redbook}, \cite{greenbook}, \cite{Lasser}) for hypergroups. However, we give below some of them in the form we need.
\subsection{Basics of semigroups} \label{lsg}\begin{itemize} \item[(i)]  As in \cite{Ross}, we consider a non-empty set $S$ linearly ordered by the relation `$<$'. \begin{itemize}

\item[(a)] For  $m,n \in S,$ we define $m \cdot n= \text{max}\{m,n\}.$ This makes it into a commutative semigroup. In this paper, we call such a semigroup $(S,<,\cdot)$ a {\it``max" semigroup}. At times, we will write it as $(S,<,\text{max}).$ Further, for $m,n \in S,$ we write $\mathcal{L}_n = \{ k \in S : k < n \},$\, $\mathcal{U}_m = \{k \in S : m <k \}.$
\item[(b)] The linear order `$<$' and the ``max" semigroup operation as in (a) above are appropriately related in the sense that  $m,n,k \in S$ and $m\leq n$ together imply that $mk\leq nk$ as specified in Definition 0.1 of linearly ordered semigroup in \cite{Hofmann}. 
\item[(c)] We may define another product `$\cdot$' by $\text{min}\{m,n\}$ but we stick to max unless needed for some explicit purpose.
\end{itemize}
\item[(ii)] We call a commutative semigroup $(S, \cdot)$ {\it max-min type} if for $m,n \in S,\, m\cdot n$ is $m$ or $n.$ Clearly, a commutative semigroup is max-min type if and only if it becomes a ``max" semigroup via : for $m,n \in S,\, m<n$ if and only if $m\cdot n=n \neq m.$
\item[(iii)] Let $(S,\cdot)$ be a semigroup with identity $e.$ For $m,n \in S$ we usually write $m\cdot n = mn.$
\begin{itemize}
\item[(a)] A non-empty subset $T$ of $S$ is called an {\it ideal} in $S$ if $TS \subset T$ and $ST \subset T,$ where $TS := \{ts \,: t \in T, s\in S\}$ and similarly for $ST.$ As in \cite{Petrich}, an ideal $T (\vneq S)$ in $S$ will be called a {\it prime ideal} if the complement $S \backslash T$ of $T$ in $S$ is a semigroup.
\item[(b)] Let $E(S)$ denote the set of idempotent elements in $S$, i.e., the set of elements $n \in S$ such that $n^2=n.$ We write $E_0(S)=E(S)\backslash \{e\}$ and $\widetilde{S}= S \backslash E(S).$
\item[(c)] Let $G(S)$ denote the set $\{g \in S: \exists\, h \in S \,\,\text{with}\,\, gh=hg=e\}.$ Then $G(S)$ is a group contained in $S$ called the {\it maximal group}. Note that $G(S) \cap E(S)= \{e\}.$  Set $G_1(S)= \{g \in G(S): \, gm=m \, \text{for all} \, m \in E_0(S)\}.$ Clearly, $G_1(S)$ is a subgroup of $G(S).$  Note that members of $G_1(S)$ act on $E_0(S)$ as the identity via left multiplication of $(S, \cdot).$ 
\end{itemize}
\end{itemize}

  \begin{defn} Let $(S,\cdot)$ be a semigroup with identity $e$. \begin{itemize}
 \item[(i)]  We call $(S, \cdot)$ {\it inverse-free} in case for $m,n \in S, mn=e$ holds if and only if $m=n=e.$ This condition is equivalent to saying that $m$ or $n$ is equal to $e.$
 
 \item[(ii)] $(S, \cdot)$ is called {\it action-free} if $G_1(S)= \{e\}.$ 
\end{itemize}
\end{defn}
 
\begin{exmp}
\begin{enumerate}
 \item[(i)] $\left(\mathbb{Z}_+,+ \right),\, \left((0,1], \times \right)$ and $ \left([1, \infty ), \times \right)$ are inverse-free semigroups which are not max-min type.
 \item[(ii)] $(\mathbb{Z}_+, \mbox{max})$ and $(\mathbb{Z}_+ \cup \{\infty\}, \mbox{min})$ are max-min type inverse-free  semigroups. 
 \item[(iii)] A commuting  set of orthogonal projections on a Hilbert space $\mathcal{H}$ containing the identity operator $I_{\mathcal{H}}$ on $\mathcal{H},$ with respect to composition of two operators, form a semigroup with identity which is  inverse-free.
 \end{enumerate} \end{exmp} 
 We know that if $(S, \cdot)$ is a commutative semigroup then $(E(S), \cdot)$ is a semigroup. Interrelations are collected in the following proposition whose proof is straight-forward.
 \begin{prop} \label{semi} Let $(S,\cdot)$ be a semigroup with identity $e.$
 \begin{itemize}
 \item[(i)] If $(S,\cdot)$ is inverse-free then $G(S)= \{e\}.$ Converse part is also true if $S$ is commutative.
 \item[(ii)] If $(S, \cdot)$ is max-min type then $(S,\cdot)$ is inverse-free.
 \item[(iii)] Suppose $S\backslash \{e\}$ is non-empty. 
 \begin{itemize}
 \item[(a)] $(S \backslash \{e\}, \cdot)$ is a semigroup if and only if  $(S,\cdot)$ is inverse-free.
 \item[(b)] If $(S, \cdot)$ is inverse-free then $(S\backslash \{e\}, \cdot)$ is an ideal in $(S, \cdot).$
 \end{itemize}

 \item[(iv)]  \label{col E0} Suppose $E_0(S)$ is non-empty. If $(S, \cdot)$ is inverse-free and commutative then $\left(E_0(S), \cdot\right)$ is a semigroup.
 \item[(v)] For a commutative semigroup $(S, \cdot),$ \,$\widetilde{S}$ is an ideal if and only if $\widetilde{S}$ is a prime ideal.
  
 \end{itemize} \end{prop}
 \begin{prop}{\bf(Dichotomy)}. \label{Dichotomy} Let $S$ be a  semigroup. For $m \in S$ either $m^j, \,j =1,2, \ldots$ are all distinct or $m^j$ is an idempotent for some $j \in \mathbb{N}.$
 \end{prop}
 We will say $m$ is of {\it infinite order} in the first case and of {\it finite order} in the second case.

 \begin{exmp} \label{2.2} Here we provide an example of a commutative semigroup $(S, \cdot)$, in which $G(S)=G_1(S) \neq \{e\}$ and thus, $S$ is not action-free. 
 
 We take $S= (\mathbb{Z}_+ \times \{0\}) \cup \{(0,1)\}$ and define `$\cdot$' as follows: \begin{eqnarray*}
 (i,0) \cdot (j,0)&=& (\text{max}\{i,j\},0) \,\,\,\text{for}\,\, i,j\in \mathbb{Z}_+,\\ (j,0) \cdot (0,1)&=& (0,1)\cdot (j,0)=(j,0)\,\, \text{for}\, j \neq 0, \\  (0,0) \cdot (0, 1)&=&(0,1)\cdot(0,0) =(0,1)\, \text{and}\\ (0,1) \cdot (0, 1)&=&(0,0). 
 \end{eqnarray*}
  Then $e= (0,0)$ and $E_0(S)= \mathbb{N} \times \{0\}$. Also $G(S)= \{(0,\alpha): \alpha =0,1 \}=G_1(S).$   
 \end{exmp}
 The semigroup in our next example is action-free but not inverse-free.
 \begin{exmp} \label{2.3} Let $S= (\mathbb{Z}_+, \text{max}) \times (\{0,1\}, \text{addition mod}\,2),$ a commutative semigroup with identity $e= (0,0).$ It is easy to see that $E_0(S)= \mathbb{N} \times \{0\},$ $G(S)= \{(0, \alpha): \alpha = 0,1\},$ and $G_1(S)= \{e\}.$
 \end{exmp}
 In this paper, we shall mainly be concerned with semigroups equipped with the discrete topology, in short, {\it discrete semigroups}.
\begin{rem} \label{nre} Let $(S,<)$ be a linearly ordered set as in Item \ref{lsg} (i).
\begin{itemize}
\item[(i)] Let $(S,<, \text{max})$ be the ``max" semigroup as in Item \ref{lsg} (i)(a) above. We equip $S$ with the order topology $\tau_0$ and assume that $S$ has an identity $e.$ Then $\tau_0$ is discrete if and only if 

(a) $e$ has an immediate successor, and, 

(b) each $m\neq e$ has an immediate predecessor as well as an immediate successor.
\item[(ii)]Two different types of examples for (i) above can be provided by 

(a) $(\mathbb{Z}_+, <, \text{max})$ or its non-empty finite subsets, and, 

(b) $\{0\} \cup \{m \pm \frac{1}{n+2}: m, n \in \mathbb{N}\}$ considered as a subset of the real line with the usual order and topology.
\end{itemize}

\end{rem} 
 
 \subsection{ Basics of hypergroups}
 Here we come to the basics of hypergroups. Dunkl \cite{Dunkl}, Jewett \cite{Jewett} and Spector \cite{Spector} independently created locally compact hypergroups (same as `convos' in \cite{Jewett}) under different names with the purpose of doing standard harmonic analysis. In this paper we are mostly concerned with commutative discrete semiconvos or hypergroups. It is convenient to write the definition in terms of a minimal number of axioms. For instance, see (\cite[Chapter 1]{Lasser}, \cite{Alagh}).

 Let $K$ be a discrete space. Let $M(K)$ be the space of complex-valued regular Borel measures on $K.$ Let $M_F(K)$ and $M_p(K)$ denote the subset of $M(K)$ consisting of measures with finite support and probability measures respectively. Let $M_{F,p}(K)= M_F(K) \cap M_p(K).$ At times, we do not distinguish between $m$ and $\delta_m$ for any $m \in K$ because $m \mapsto \delta_m$ is an embedding from $K$ into $M_p(K).$ Here $\delta_m$ is the unit point mass at $m,$ i.e., the Dirac-delta measure at $m.$ 
  
   We begin with a map $* : K \times K \rightarrow M_{F,p}(K)$. Simple computations enable us to extend `$*$' to a bilinear map called {\it convolution}, denoted by `$*$' again, from $M(K) \times M(K)$ to $M(K).$ At times, for certain $n \in K$ we will write $q_n$ for $ \delta_n* \delta_n$ and $Q_n$ for its support.
   
    A bijective map $\vee: m \mapsto \check{m}$ from $K$ to $K$ is called an {\it involution} if $\check{\check{m}}=m.$ We can extend it to $ M(K)$ in a natural way.  
   
 \begin{defn} \label{semiconvo} A pair $(K,*)$ is called a { \it discrete semiconvo} if the following conditions hold.
 \begin{itemize}
 \item The map $* : K \times K \rightarrow M_{F,p}(K)$ satisfies the associativity condition $$ (\delta_m*\delta_n)*\delta_k = \delta_m* (\delta_n*\delta_k)\,\,\, \text{for all}\, m, n, k \in K.$$
 \item There exists (necessarily unique) element $e \in K$ such that $$ \delta_m*\delta_e = \delta_e*\delta_m = \delta_m \,\,\,\,\text{for all}\,\, m \in K.$$ 
 \end{itemize}\end{defn}
 A discrete semiconvo $(K,*)$ is called {\it commutative} if $\delta_m*\delta_n= \delta_n*\delta_m$ for all $m,n \in K.$
 
 \begin{defn} A triplet $(K,*, \vee)$ is called a {\it discrete hypergroup} if 
 \begin{itemize}
 \item $(K,*)$ is a discrete semiconvo,
 \item  $\vee$ is an involution on $K$ that satisfies \begin{itemize}
 \item[(i)]  $(\delta_m*\delta_n \check{)}= \delta_{\check{n}}*\delta_{\check{m}}$ for all $m,n \in K$ and 
 \item[(ii)] $e \in \spt(\delta_m*\delta_{\check{n}})$ if and only if $m=n.$
\end{itemize}  
 \end{itemize} \end{defn}
 A discrete hypergroup $(K,*, \vee)$ is called {\it hermitian} if the involution on $K$ is the identity map, i.e., $\check{m}=m$ for all $m \in K.$ 
 
  Note that a hermitian discrete hypergroup is commutative.
  
   We write $(K,*)$ or $(K,*,\vee)$ as $K$ only if no confusion can arise. 

 Let K be a commutative discrete hypergroup. For a complex-valued function $\chi$ defined on $K,$  we write $\check{\chi}(m):= \overline{\chi(\check{m})}$ and $\chi(m*n) = \int_K \chi\, d(\delta_m*\delta_n)$ for $m,n \in K.$ Now, define two dual objects of $K:$
 $$ \mathcal{X}_b(K)= \left\lbrace\chi \in \ell^\infty(K): \chi \neq 0, \chi(m*n)= \chi(m) \chi(n) \, \text{for all}\,\, m,n \in K \right\rbrace,$$ $$\widehat{K}= \left\lbrace \chi \in \mathcal{X}_b(K) : \check{\chi}= \chi,\,\mbox{i.e.,}\, \chi(\check{m})= \overline{\chi(m)}\,\, \text{for all} \,m \in K  \right\rbrace.$$ Each $\chi \in \mathcal{X}_b(K)$ is called a {\it character} and each $\chi \in \widehat{K}$ is called a {\it symmetric character}. With the topology of pointwise convergence, $\mathcal{X}_b(K)$ and $\widehat{K}$ become compact Hausdorff spaces. In contrast to the group case, these two dual objects need not be the same and also need not have a hypergroup structure. 
 
 Now we give some examples of hypergroups.
\begin{exmp} {\bf Polynomial hypergroups:} \label{2.4}This is a wide and important class of hermitian discrete hypergroups in which hypergroup structures are defined on $\mathbb{Z}_+.$ This class contains Chebyshev polynomial hypergroups of first kind, Chebyshev polynomial hypergroups of second kind, Jacobi hypergroups, Laguerre hypergroups etc. For more details see \cite{Bloom} and \cite{Lasser}. 

For Illustration, we describe CP, the Chebyshev polynomial hypergroup of first kind which arises from the Chebyshev polynomials of first kind. In fact, they define the following convolution `$*$' on $\mathbb{Z}_+:$
$$\delta_m*\delta_n= \frac{1}{2} \delta_{|n-m|}+\frac{1}{2} \delta_{n+m} \,\,\,\text{for}\,\, m,n \in \mathbb{Z}_+.$$  For any $k \in \mathbb{N},$ $K= k \mathbb{Z}_+$ with $*|_{K \times K}$ makes $K$ a discrete hypergroup in its own right. 
\end{exmp} The Chebyshev polynomial hypergroup of second kind $(\mathbb{Z}_+, *)$ arises from the Chebyshev polynomials of second kind and the convolution '$*$' on $\mathbb{Z}_+$ is given by $$ \delta_m*\delta_n= \sum_{k=0}^{\text{min}\{m ,n\}} \frac{|m-n|+2k+1}{(m+1)(n+1)} \delta_{|m-n|+2k}.$$

\begin{exmp} \label{Dk} Let $H_a=\{0,1,2, \ldots, \infty\}, 0<a \leq \frac{1}{2},$ be the one-point compactification of $\mathbb{Z}_+.$ Dunkl and Ramirez \cite{Ramirez} defined a convolution structure on $H_a$ to make it a (hermitian) countable compact hypergroup. For a prime $p,$ let $\Delta_p$ be the ring of p-adic integers and $\mathcal{W}$ be its group of units, that is , $\{x=x_0+x_1p+ \ldots+ x_np^n+ \ldots \in \Delta_p : x_j = 0,1, \ldots,p-1 \, \text{for}\, j \geq 0 \, \text{and} \, x_0 \neq 0  \}$. For $a=\frac{1}{p},$ $H_a$ derives its structure from $\mathcal{W}$-orbits of action of $\mathcal{W}$ on $\Delta_p$ by multiplication in $\Delta_p.$ 
\end{exmp} 
Next, Dunkl and Ramirez make the symmetric dual space $\widehat{H_a}$ of $H_a$ into a hermitian discrete hypergroup. The members of $\widehat{H_a}$ are given by $\{\chi_n : n \in \mathbb{Z}_+\},$ where, for $k \in H_a,$ \begin{eqnarray*}
\chi_n(k)= \begin{cases}  0 & \text{if}\,\,k <n-1, \\ \frac{a}{a-1} & \text{if}\,\,k=n-1, \\ 1 &\text{if}\,\, k \geq n \,\,\, (\text{or}\,\, k = \infty).\end{cases}
\end{eqnarray*}
Then the convolution `$*$' on $K=\mathbb{Z}_+$ identified with $\widehat{H_a}=\{\chi_n : n \in \mathbb{Z}_+\}$ is dictated by pointwise product of functions in $\widehat{H_a},$ that is: 
\begin{eqnarray*}
\chi_m \chi_n &=& \chi_{\text{max}\{m,n\}} \,\,\, \text{for}\,\, m \neq n, \\
\chi_0^2&=&\chi_0, \,\,\,\, \chi_1^2 = \frac{a}{1-a} \chi_0+ \frac{1-2a}{1-a} \chi_1,\\
\chi_n^2&=& \frac{a^n}{1-a} \chi_0+ \sum_{k=1}^{n-1} a^{n-k} \chi_k+\frac{1-2a}{1-a} \chi_n \,\,\,\,\, \text{for}\, n \geq 2.
\end{eqnarray*}

   We call $(K,*)$ a {\it (discrete) Dunkl-Ramirez hypergroup}.
   
    $H_a$ has a good spectral synthesis in the sense that every closed subset of $H_a$ is a set of spectral synthesis for the  Fourier algebra $A(H_a)$ \cite[Theorem 10.6]{Ramirez}. Chilana (now, Ajit Iqbal Singh) and Ajay Kumar \cite{Ajit} have strengthened this further.
\begin{rem} \label{Drem} We note a few more standard facts about a Dunkl-Ramirez hypergroup $(K,*).$ We follow the notation as in $2.1$ and $2.2$ above.
\begin{itemize}
\item[(i)] For $n \in \mathbb{Z}_+,$ $(\mathcal{L}_n \cup \{n\}, *)$ is a subhypergroup of $(K,*).$
\item[(ii)] Let $a \neq \frac{1}{2}.$ For $n \in \mathbb{Z}_+,\, Q_n= \mathcal{L}_n \cup \{n\}$ and as a consequence, for $n \in \mathbb{N},\, \#Q_n \geq 2.$
\item[(iii)] Let $a= \frac{1}{2}.$ 
\begin{itemize}
\item[(a)] For $n \in \mathbb{N}, \,Q_n= \mathcal{L}_n.$ For $n \geq 2,\, \#Q_n \geq 2$ whereas for $n=1, \, \#Q_n=1.$ 
\item[(b)] In fact, $(\mathcal{L}_1 \cup \{1\},*)$ is a group isomorphic to $(\{0,1\}, \text{addition mod}\, 2)$ 
\end{itemize}

\end{itemize}
\end{rem}  
 
  \section{Hypergroups arising from hypergroup deformations of idempotent elements of ``max" semigroups} 
 To begin with, we assume that $(S,\cdot)$ is a commutative  discrete semigroup with identity $e$ such that $E_0(S) \neq \phi.$ For $q \in M_p(S),$  let $Q$ be its support, in short, $\spt(q)$ and $q(j)= q(\{j\})$ for $j \in S.$ Then $Q$ is countable. Observe that $q=\sum_{j \in Q} q(j) \delta_j$ with $q(j)>0$ for each $j \in Q$ and $\sum_{j \in Q} q(j)=1.$ We prefer this to the usual form $\sum_{j \in S} q(j) \delta_j$ with $q(j) \geq 0$ for $j \in S$ and $\sum_{j \in S} q(j)=1,$ unless otherwise stated.
 \begin{defn} A probability measure $q$ on $S$ with $\#Q \geq 2$ will be called \it{non-Dirac.}
 \end{defn}
 \subsection{Motivation} \label{moti}
 We note that a Dunkl-Ramirez hypergroup (as in Example \ref{Dk} above) is a hermitian (hence commutative)  discrete hypergroup $K=\widehat{H_a}\,\, (0<a\leq \frac{1}{2})$ and its convolution `$*$' arises as a hypergroup deformation of the semigroup $(\mathbb{Z}_+, \cdot),$ where $m \cdot n= \mbox{max}\{m, n\}$ in the sense that $\delta_m*\delta_n= \delta_{mn}$ for $m \neq n$ or, $m=n=0$ and for $m=n \neq 0,$ we have
 \begin{eqnarray*}
 \delta_1*\delta_1 &=& \frac{a}{1-a} \delta_0+\frac{1-2a}{1-a}\delta_1,\\
 \delta_n*\delta_n &=&  \frac{a^n}{1-a} \delta_0+ \sum_{k=1}^{n-1} a^{n-k} \delta_k + \frac{1-2a}{1-a} \delta_n \,\,\,\,\, \text{for}\,\, n \geq 2.
 \end{eqnarray*}

 Also, each element of $\mathbb{Z}_+$ is an idempotent in $(\mathbb{Z}_+, \cdot).$ Further, for $n \in \mathbb{Z}_+,$ $(\mathcal{L}_n \cup \{n\}, *)$ arises as a hypergroup deformation of the  finite subsemigroup $(\mathcal{L}_n \cup \{n\}, \text{max})$ of $(\mathbb{Z}_+, \cdot).$ This motivates the construction of a  commutative discrete  semiconvo or discrete hypergroup structure on $S$ for a commutative  discrete semigroup $(S, \cdot)$ by deforming the product. We elaborate as follows.
 
 \subsection{Hypergroups deformations of  ``max" semigroups } \label{avi} We now confine our attention to  a ``max" semigroup $(S, <, \cdot)$ as in Item \ref{lsg}(i)(a). We assume that $S$ has an identity $e$ and it is equipped with the discrete topology.
 We try to deform this discrete semigroup $(S, \cdot)$ into a  hermitian (hence commutative) discrete hypergroup by deforming the product on the diagonal of $S \backslash \{e\}.$ 
\subsubsection{{\bf Preparatory material}} \label{prem} (i) For a discrete space $K$ and a complex-valued regular Borel or a non-negative measure $\mu$ on $K,$ we usually write $\mu(j)$ for $\mu(\{j\}).$

(ii)(a) Let $(K,*)$ be a discrete hermitian hypergroup. Then by \cite[Theorem 7.1A]{Jewett} the Haar measure $\lambda$ on $K$ is given by: $\lambda(e)=1 $ and for $e \neq n \in K,$ $\lambda(n)= \frac{1}{(\delta_n*\delta_n)(e)}.$  

 (b) Let $K$ be a Dunkl-Ramirez hypergroup $ \widehat{H_a} \,(0< a \leq \frac{1}{2}).$ Then $\lambda(n)= \frac{1-a}{a^n}$ for all $n \in \mathbb{N}.$ Further, for $n \in \mathbb{N},$  we have  
\begin{eqnarray*}
\delta_n*\delta_n(0) &=& \frac{a^n}{1-a}= \frac{1}{\lambda(n)},\\ \delta_n*\delta_n(k)&=& a^{n-k} =
\frac{\lambda(k)}{\lambda(n)} \,\,\text{for}\,\,1\leq k <n \,\, \text{and}\\ \delta_n*\delta_n(n)&=& \frac{1-2a}{1-a}= 1- \sum_{0 \leq k<n} (\delta_n*\delta_n)(k)= 1- (\delta_n*\delta_n)(\mathcal{L}_n)= \frac{\lambda(n)-\lambda(\mathcal{L}_n)}{\lambda(n)}.
\end{eqnarray*} 
(c) $(\delta_n*\delta_n)(n)=0$ for some $n \in \mathbb{N}$ if and only if $a=\frac{1}{2}$ if and only if $(\delta_n*\delta_n)(n)=0$ for all $n \in \mathbb{N}.$ In this case, $\lambda(n)=2^{n-1}$ for $n \in \mathbb{N}.$

{\setlength{\parskip}{1.5em}

We try to replace $(\mathbb{Z}_+,<, \cdot)$ in the Dunkl-Ramirez hypergroups by a discrete ``max" semigroup $(S,<, \text{max})$ with identity in our next theorem.}

We give a set of necessary and sufficient conditions for $(S,*)$ to become a  hermitian discrete hypergroup with convolution product `$*$' defined as follows: 
 \begin{eqnarray*}
 \delta_m*\delta_n = \delta_n*\delta_m &=& \delta_{m\cdot n} (= \delta_{\text{max}\{m,n\}}) \,\,\,\,\mbox{for}\, m,n \in S\, \text{with}\, m \neq n, \, \text{or},\, m=n=e,\\ \delta_n*\delta_n &=& q_n \,\,\,\,\,\,\,\,\,\,\, \text{for} \, n \in S\backslash \{e\}.
 \end{eqnarray*}
  Here, $q_n$ is a probability measure on $S$ with finite support $Q_n$ containing $e$ and has the form $\sum_{j \in Q_n} q_n(j) \delta_j$ with $q_n(j)>0$ for $j \in Q_n$ and $\sum_{j \in Q_n} q_n(j)=1.$
  
   This is equivalent to looking for conditions on $S,$ $Q_n$'s, an $S$-tuple $\{v_n\}_{n \in S}$ in $[1, \infty)$ with $v_e=1$ and $v_n= \frac{1}{q_n(e)}$ for $n \in S \backslash \{e\}$  and $q_n(j)$ for $j \in Q_n \backslash \{e\}, \,n \in S \backslash\{e\}.$
  
  
 \begin{thm} \label{Ordered} Let $(S,<, \cdot)$ be a discrete (commutative) ``max" semigroup with identity $e$ and ` $*$' and other related symbols as above. Then $(S,*)$ is a hermitian discrete hypergroup if and only if the following conditions hold. 
 \begin{itemize}
 \item[(i)] Either $S$ is finite or $(S,<,\cdot)$ is isomorphic to $(\mathbb{Z}_+,<, \text{max}).$
 \item[(ii)] For $n \in S  \backslash \{e\},$ we have $\mathcal{L}_n \subset Q_n \subset \mathcal{L}_n \cup \{n\}.$ 
 \item[(iii)]  If $\#S > 2,$ then for $ e \neq m <n $ in $S,$ we have 
 \begin{itemize}
 \item[(a)] $q_n(e)=q_n(m)q_m(e)$ and 
 \item[(b)] $q_n(e) \left( 1+ \sum_{e \neq k \in \mathcal{L}_n} \frac{1}{q_k(e)} \right) \leq 1;$ 
 \end{itemize} 
 or, equivalently, with $v_n= \frac{1}{q_n(e)}$ for $n \in S,$
 \item[(iii)']  If $\#S > 2,$ then for $ e \neq m <n $ in $S,$ we have 
 \begin{itemize}
 \item[(a)]$q_n(m)= \frac{v_m}{v_n}$ and 
 \item[(b)] $\sum_{k \in \mathcal{L}_n}v_k \leq v_n.$ 
 \end{itemize} \end{itemize} \end{thm}
 \begin{proof}
Suppose $(S,*)$ is a hermitian discrete hypergroup. 

We prove (i) in two steps.

{\bf Step ($\alpha$)}. Let $n \in S \backslash \{e\}$ and $m<n$ . Associativity of $(S, *)$ demands that $(\delta_n*\delta_n)*\delta_m=\delta_n*(\delta_n*\delta_m).$ This gives that $q_n*\delta_m=q_n.$ But $e \in Q_n;$ therefore, $m \in Q_n.$ Hence $\mathcal{L}_n \subset Q_n.$ Since $\#Q_n < \infty$ we have $\#\mathcal{L}_n < \infty.$  Clearly, $\mathcal{L}_e= \phi$ is finite.

{\bf Step ($\beta$)}. Assume that $S$ is not finite. Consider any $n \in S$. By Step ($\alpha$) $\mathcal{L}_n \cup \{n\}$ is finite and, therefore,  $\mathcal{U}_n = S \backslash (\mathcal{L}_n \cup \{n\})$ is non-empty. Choose any $m \in \mathcal{U}_n$.  Suppose that $m$ is not immediate successor of $n.$  Then, there exists an element $m' \in S$ such that $n<m'<m.$ So $\mathcal{U}_n \cap \mathcal{L}_m$ is finite and non-empty because $m' \in \mathcal{U}_n \cap \mathcal{L}_m $ and $\mathcal{L}_m$ is finite by Step $(\alpha)$. Since $\mathcal{U}_n \cap \mathcal{L}_m$ is finite we can enumerate its elements in  order, say, $n_1<n_2< \cdot \cdot \cdot <n_s.$  Choose the immediate successor $n_1$ of $n.$ Thus each element $n$ of $S$ has an immediate successor, say, $s(n)$. Now, start with $x_0=e,$
take $x_1= s(e), \,x_2=s(x_1), \cdot \cdot\cdot, x_j=s(x_{j-1})$ and so on.  
Set $W=\{x_j:\, j \in \mathbb{Z}_+\}.$ We claim that $W=S.$ Suppose not; then there exists $t \in S\backslash W.$ Therefore, $t \neq e,$ and hence $t>e.$ But $x_1$ is immediate successor of $e$ so $t>x_1.$ Again, $x_2$ is the immediate successor of $x_1,$ so $t>x_2.$ Repeating this process, we get $t>x_j$ for every $j \in \mathbb{Z}_+.$ Hence $W \subset \mathcal{L}_t.$ This shows that $W$ is finite because $\mathcal{L}_t$ is finite by Step $(\alpha)$. This gives a contradiction. So, $S=W.$ 

 Therefore, either $S$ is finite or $(S,<,\cdot)$ is isomorphic to $(\mathbb{Z}_+,<, \text{max}).$

(ii) Consider any $n \in S \backslash \{e\}$. Let, if possible,  $Q_n \not\subset \mathcal{L}_n \cup \{n\}$. Then there exists $ m >n$ with $m \in Q_n$. The associativity of `$*$' demands that $(\delta_n*\delta_n)*\delta_m= \delta_n*(\delta_n*\delta_m),$  which in turn gives that $q_n*\delta_m=\delta_m.$ Now, L.H.S. $= q_n(m) \delta_m*\delta_m + \sum_{m \neq j \in Q_n} q_n(j) \delta_{jm}=q_n(m) q_m + \sum_{m \neq j \in Q_n} q_n(j) \delta_{jm}.$ By Step (i) $(\alpha), \, \mathcal{L}_m \subset Q_m.$ Now, $Q_m \subset \spt(\text{L.H.S.})=\spt(\text{R.H.S.})= \{m\}.$ But $n \in \mathcal{L}_m.$ Thus, we have $n=m,$ which is a contradiction. Therefore, $Q_n \subset \mathcal{L}_n \cup \{n\}.$ Hence, using Step (i)$(\alpha),$ we get $\mathcal{L}_n \subset Q_n \subset \mathcal{L}_n \cup \{n\}$ for $n \in S \backslash \{e\}.$

(iii) Suppose $\#S>2$ and $ e \neq m <n $ in $S.$ 
\begin{itemize}
\item[(a)] As in the proof of Step(i)$(\alpha),\,$ 
 $q_n*\delta_m= q_n.$ Now, as by (ii), $m \in Q_n,$ we get 
\begin{eqnarray*}
\mbox{L.H.S.} &=&  q_n(e) \delta_m + \sum_{\underset{e \neq j \neq m} {j \in Q_n}} q_n(j) \delta_{jm} + q_n(m) \delta_m*\delta_m \\ &=& q_n(e) \delta_m + \sum_{\underset{e \neq j \neq m} {j \in Q_n}} q_n(j) \delta_{jm} + q_n(m) \sum_{j' \in Q_m} q_m(j') \delta_{j'}\,\, \text{and}
\end{eqnarray*} 
\begin{eqnarray*}
\mbox{R.H.S.} = q_n(e) \delta_e + \sum_{e \neq k \neq m, \, k \in Q_n} q_n(k) \delta_k + q_n(m) \delta_m. 
\end{eqnarray*}
Since $jm \neq e$ for $j,m \in S \backslash \{e\}$ and R.H.S. = L.H.S., we get $ q_n(e)= q_n(m) \,q_m(e).$ 

\item[(b)] 
 Because $n \in S \backslash \{e\},$ we get $  0< \sum_{ j \in \mathcal{L}_n} q_n(j) \leq \sum_{j \in Q_n} q_n(j) =1.$ By (iii)(a), $q_n(j)= \frac{q_n(e)}{q_j(e)}$ for all $ j \in \mathcal{L}_n \backslash \{e\}.$ Therefore, we get $$ q_n(e)\left( 1+ \sum_{e \neq j \in \mathcal{L}_n} \frac{1}{q_j(e)}\right) = q_n(e) +\sum_{e\neq j \in \mathcal{L}_n} \frac{q_n(e)}{q_j(e)}= \sum_{j\in \mathcal{L}_n} q_n(j) \leq 1.$$
\end{itemize}
It can be seen easily that (iii)' is only a rewording of (iii).

For the  converse part, assume that (i)-(iii) all hold. We note right at the outset that, because condition (ii) is satisfied, for $n \in S\backslash \{e\},$ we may write $q_n=\delta_n*\delta_n= \sum_{j \in \mathcal{L}_n \cup \{n\}} q_n(j) \delta_j= \sum_{j \leq n} q_n(j) \delta_j$ with $q_n(j) >0$ for $j \in \mathcal{L}_n$ and $q_n(n) \geq 0.$ 

 To see that $(S,*)$ is a hermitian  discrete hypergroup, first note that the identity element $e$ of semigroup $(S, \cdot)$ works as the identity of $(S,*),$ i.e., $\delta_n*\delta_e= \delta_e*\delta_n=\delta_n$ for every $n \in S.$ 

Next, we show that $e \in \spt(\delta_m*\delta_n)$ if and only if $m=n.$ It is trivial in case $m=n=e.$  If $m =n \neq e $ then $e \in \spt(\delta_n*\delta_n)= Q_n.$ Now, assume that $e \in \spt(\delta_m*\delta_n).$ Let, if possible, $m \neq n,$ so $\delta_m*\delta_n= \delta_{mn}.$ Without loss of generality we can assume that $m<n,$ so $\delta_{mn}=\delta_n$ and $n>e.$ Therefore, $e \in \spt(\delta_m*\delta_n)$ shows that $n =e,$ which is a contradiction. Hence $m=n.$

Now, we prove the associativity condition of $(S,*)$, i.e., for $m,n,k \in S$
\eqspecialnum \begin{equation} \label{Com}
(\delta_m*\delta_n)*\delta_k= \delta_m*(\delta_n*\delta_k).
\end{equation}
If any one or more of $\{m,n,k\}$ are $e$ then \eqref{Com} is trivial. We now consider the cases when none of $m,n$ and $k$ are $e.$ 
\begin{itemize}
\item[Case (i)] If $m \neq n \neq k \neq m,$ we first consider the subcase $m<n<k.$ So, L.H.S. of \eqref{Com} $= \delta_n*\delta_k =\delta_k$ and R.H.S. of \eqref{Com} $= \delta_m*\delta_k=\delta_k.$ Therefore, L.H.S. =R.H.S. Other subcases can be dealt with in a similar way and hence \eqref{Com} holds. 

 \item[Case (ii)] Out of $ m,n,k$ exactly two are distinct. We get $\#S>2.$
\begin{itemize}
\item[($\alpha$)] We consider first the following form of \eqref{Com}:  $ (\delta_n*\delta_n)*\delta_m= \delta_n*(\delta_n*\delta_m)$ with $e \neq n \neq m \neq e,$ which is the same as 
 \eqspecialnum \begin{equation} \label{Dom}
q_n*\delta_m = \delta_n*\delta_{nm}
\end{equation}
 When $n<m,$ R.H.S. of \eqref{Dom} is equal to $\delta_n*\delta_m=\delta_m.$ The condition that $Q_n \subset \mathcal{L}_n \cup \{n\}$ implies that $Q_n\cdot m =\{m\}$ and using this, we get  L.H.S. of \eqref{Dom} $= \sum_{j \in Q_n} q_n(j) (\delta_j*\delta_m)= \sum_{j \in Q_n} q_n(j) \delta_{jm}=\left( \sum_{j \in Q_n} q_n(j) \right)\delta_m =\delta_m$ as $\sum_{j \in Q_n} q_n(j) =1 .$ Therefore, L.H.S.= R.H.S. and hence \eqref{Dom} holds.

We now come to the case when $m<n.$ Then R.H.S. of \eqref{Dom} is $q_n= \sum_{j \leq n} q_n(j) \delta_j$ with $q_n(j)>0$ for $j <n$ and $q_n(n) \geq 0.$ 
  So L.H.S. of \eqref{Dom} 
\begin{eqnarray*}
 &=& \sum_{j \leq n} q_n(j) (\delta_j*\delta_m)= \sum_{\underset {j \neq m}{j\leq n}} q_n(j) \delta_{jm}+ q_n(m) \sum_{k \leq m} q_m(k) \delta_k \\  &=& \sum_{m<j \leq n } q_n(j) \delta_j + \left( \sum_{j <m} q_n(j) \right) \delta_m + \sum_{k<m} q_n(m) q_m(k) \delta_k + q_n(m) q_m(m)\delta_m .
\end{eqnarray*}

By (iii)(a) $ q_n(k)\, q_k(e)= q_n(e)$ for $k<n.$ So we get $ q_n(m)\, q_m(j)= q_n(j)$ for all $j <m.$ Therefore, L.H.S. of \eqref{Dom}
\begin{eqnarray*}
 &=&  \sum_{m<j \leq n} q_n(j) \delta_j + \left( \sum_{j<m} q_n(m) q_m(j) +  q_n(m) q_m(m)\right) \delta_m+  \sum_{k <m} q_n(k) \delta_k \\ &=&  \sum_{m<j \leq n } q_n(j) \delta_j +  q_n(m) \left( \sum_{j \leq m} q_m(j)  \right) \delta_m+  \sum_{k<m} q_n(k) \delta_k.
\end{eqnarray*}
Further, we have $\sum_{j \leq m} q_m(j) =1.$ So,
$\text{L.H.S. of \eqref{Dom} } =  \sum_{k \leq n} q_n(k) \delta_k = q_n. $

\noindent  Therefore, we get L.H.S. of \eqref{Dom} = R. H.S. of \eqref{Dom} and hence \eqref{Dom} holds.
\item[($\beta$)] Now, consider the following form of \eqref{Com} with $m<n,$ $(\delta_n*\delta_m)*\delta_n= \delta_n*(\delta_m*\delta_n).$ R.H.S. $= \delta_n*\delta_n = (\delta_n*\delta_m)*\delta_n=$ L.H.S. 
\item[($\gamma$)] The case $m>n$ of the form \eqref{Com} as in $(\beta)$ above follows simply because L.H.S. $= \delta_m*\delta_n= \delta_m=$ R.H.S.
\end{itemize}

\item[(iii)] If $m=n=k$ then \eqref{Com} follows from the fact $\delta_s*\delta_n= \delta_n*\delta_s$ for $s \in Q_n.$
\end{itemize}
Hence \eqref{Com} holds and consequently $(S,*)$ is a hermitian discrete hypergroup.
\end{proof}

\begin{rem} \label{b_n} In view of condition (i) in Theorem \ref{Ordered} and Remark \ref{nre}(i) we note that for ``max" semigroups $(S,<, \text{max})$ of interest to us the discrete topology and the order topology coincide (cf. \cite{Zuck}, \cite{Ross}).
\end{rem} 

We now interpret conditions (i)-(iii) in Theorem \ref{Ordered} above. Because a finite ``max" semigroup is isomorphic to $(\mathbb{Z}_k= \{j: 0 \leq j < k\}, <, \text{max})$ for some $k \in \mathbb{N},$ we confine our attention to $\mathbb{Z}_+$ to begin with. 

Let $\mathcal{V}$ be the set of sequences $v=(v_j)_{j \in \mathbb{Z}_+}$ in $[1, \infty)$ which satisfy (i) $v_0=1$ and (ii) $v_n \geq \sum_{j \in \mathcal{L}_n}v_j$ for $n \in \mathbb{N}.$ 

For $n \in \mathbb{N},$ let $u_n= v_n- \sum_{j \in \mathcal{L}_n}v_j.$ Then simple calculations give the following:
 \begin{eqnarray*} 
v_0 &=& 1 , \\ v_1&=&1+u_1,\\ v_2&=&(2+u_1)+u_2, \\ v_3 &=& (2^2+2u_1+u_2)+u_3, \\ \vdots &=& \vdots, \\ \text{to elaborate},\\ v_n &=& (2^{n-1}+2^{n-2}u_1+ \ldots+u_{n-1})+u_n \,\,\,\ \text{for}\,\,n \geq 3.
\end{eqnarray*} 

Alternatively, we may start with a sequence $u=(u_n)_{n \in \mathbb{N}}$ in $[0, \infty)$ and define $(v_n)_{n \in \mathbb{Z}_+}$ as above. Let $\mathcal{U}$ be the set of such sequences $(u_n)_{n \in \mathbb{N}}$'s. The sets $\mathcal{U}$ and $\mathcal{V}$ have cardinality $\mathfrak {c}.$ 

 The same is true if we confine our attention to domain $\mathbb{Z}_k= \{j: 0 \leq j < k\}$ for $k \in \mathbb{N},$ instead of $\mathbb{Z}_+.$
 
 
 \begin{cor} \label{coro} For each $v \in \mathcal{V}$ (or the corresponding $u \in \mathcal{U}$), there is one and only one (hermitian) hypergroup deformation $(\mathbb{Z}_+,*)$ of $(\mathbb{Z}_+,<, \text{max})$ which satisfies $(\delta_n*\delta_n)(0)= \frac{1}{v_n},$ $n \in \mathbb{Z}_+.$ Further, for this deformation, the convolution `$*$' and the Haar measure $\lambda$ satisfy the following conditions.
  \begin{itemize}
 \item[(i)] $ \lambda(n)=v_n$ for $n \in \mathbb{Z}_+$ and  $\lambda(n)-\lambda(\mathcal{L}_n)=u_n$ for $n \in \mathbb{N}.$
 \item[(ii)] $\lambda(\mathcal{L}_n) \leq \lambda(n)$ for $n \in \mathbb{N}.$
 \item[(iii)] For $n \in \mathbb{N},$ 
 \begin{itemize}
 \item[(a)] $\delta_n*\delta_n(m)= \frac{\lambda(m)}{\lambda(n)}$ for $m<n,$
 \item[(b)] $\delta_n*\delta_n(n)=\frac{\lambda(n)-\lambda(\mathcal{L}_n)}{\lambda(n)},$
 \item[(c)] $\delta_n*\delta_n(m)=0$ for $m>n.$
 \end{itemize}
 \item[(iv)] For $n \in \mathbb{N},$
$\spt(\delta_n*\delta_n) = \begin{cases} \mathcal{L}_n & \text{if}\,\, \lambda(n)=\lambda(\mathcal{L}_n), \\ \mathcal{L}_n \cup \{n\} & \text{if}\,\, \lambda(n)>\lambda(\mathcal{L}_n). 
\end{cases}$
 \end{itemize}
 \end{cor} 
 \begin{proof}
 (i) This is immediate from \ref{prem}(ii)(a) and expressions for $(v_n)_{n\in \mathbb{Z}_+}$ and $(u_n)_{n \in \mathbb{N}}$ given above.
 
 (ii) This follows from Theorem \ref{Ordered}(iii)(b).
 
 (iii) We have only to appeal to Theorem \ref{Ordered}(iii)(a).
 
 (iv) This follows from (iii).
 \end{proof}
 
 \begin{rem} We compare Theorem \ref{Ordered} and Corollary \ref{coro} above with the Dunkl-Ramirez hypergroups (Example \ref{Dk} above). We freely use Remark \ref{Drem}, Section \ref{moti} and Item \ref{prem} above.
  \begin{itemize}
\item[(a)] For the Dunkl-Ramirez hypergroup $ \widehat{H_a} \,(0< a \leq \frac{1}{2})$, $q_n(0)= \frac{a^n}{1-a}, \, n \in \mathbb{N}.$ Further, $Q_m= \mathcal{L}_m$ for some $m \in \mathbb{N}$ if and only if $a = \frac{1}{2}$ if and only if $Q_n=\mathcal{L}_n$ for all $n \in \mathbb{N}.$
\item[(b)] For our theorem, for any arbitrarily fixed $m \in S \backslash \{e\},Q_m=\mathcal{L}_m$  if and only if \newline $q_m(e) \left(1+\sum_{e\neq j \in \mathcal{L}_m} \frac{1}{q_j(e)} \right)=1.$ Under the correspondence set up after Theorem \ref{Ordered}, the latter condition  can be shortened to $u_m =0$. 
\item[(c)] Suppose $\#S>2.$ For any strictly increasing sequence $(n_j)$ in $S \backslash \{e\}$ contained in $S,$ considered as a subsemigroup of $(\mathbb{Z}_+,<, \text{max})$ in view of Theorem \ref{Ordered} (i) above, we can find $(q_n(e))_{n \in S}$ leading to a deformation satisfying $Q_{n_j}= \mathcal{L}_{n_j}\, \forall j$ and $Q_n= \mathcal{L}_n \cup \{n\}$ for $n \in S \backslash \{e\}$ but not equal to any $n_j.$ 
\end{itemize}
 \end{rem}
\subsection{The dual objects of $(S,*)$} 
 
We now come to the dual objects $\mathcal{X}_b(S)$ and $\widehat{S}$ of hypergroup $(S,*),$ as in Theorem \ref{Ordered} and Corollary \ref{coro} above. For this purpose, to begin with, we fix any $v \in \mathcal{V}$ (or the corresponding $u \in \mathcal{U}$) and consider the corresponding deformation $(\mathbb{Z}_+,*)$ as obtained in Theorem \ref{Ordered} and Corollary \ref{coro}. We freely follow the concepts, results and notation set up in this process.

\begin{thm} \label{St} Let $(S,<, \cdot) \cong (\mathbb{Z}_+, <, \text{max})$ and other symbols  satisfy the conditions (ii)-(iii) of Theorem \ref{Ordered} and let $(S,*)$ be the corresponding deformed hypergroup with the Haar measure $\lambda$. Then the dual objects  $\mathcal{X}_b(S)$ and $\widehat{S}$ of $(S,*)$ are equal. Equipped with the topology of uniform convergence on compact subsets of $S,$ $\widehat{S}$ can be identified with the one point compactification $\mathbb{Z}_+^* (= \mathbb{Z}_+ \cup \{\infty\})$ of $\mathbb{Z}_+.$  More precisely, the identification is given by $k \mapsto \mychi_k,$ where $\mychi_\infty(n) =1 $ for all $n \in \mathbb{Z}_+,$ and, for $k \in \mathbb{Z}_+,$ $\mychi_k$ is given by 
$$\mychi_k(n) = \begin{cases}1 &\text{if}\,\, n \leq k, \\ \beta_k & \text{if}\,\, n=k+1, \\  0 & \text{if}\,\,n>k+1,
\end{cases}$$ 
where, $\beta_k= \frac{-\lambda(\mathcal{L}_{k+1})}{\lambda(k+1)}= \frac{-\sum_{j \in \mathcal{L}_{k+1}}v_j}{v_{k+1}} =\frac{u_{k+1}}{v_{k+1}}-1= \delta_{k+1}*\delta_{k+1}(k+1)-1=q_{k+1}(k+1)-1.$
 \end{thm}
 
 \begin{proof} Before we start the proof, it is helpful to give a diagrammatic interpretation (Table \ref{tabamc}) of the functions $\mychi_k$'s, $k \in \mathbb{Z}_+$. As in the proof of Theorem \ref{Ordered}, for $n \in \mathbb{N},$ we may write $q_n= \sum_{j=0}^n q_n(j) \delta_j$ with $q_n(j)>0$ for $0 \leq j<n,$ but $q_n(n) \geq 0$ and $\sum_{j=0}^n q_n(j)=1.$
 
 We note that, for $k \in \mathbb{Z}_+,$ $\beta_k=q_{k+1}(k+1)-1$ and $-1 \leq \beta_k <0.$

    \begin{table}[h!]
\renewcommand{\kbldelim}{(}
\renewcommand{\kbrdelim}{)}
\[
   \kbordermatrix{
    & 0 & 1 & 2 & 3 &\cdots & k & k+1&  \cdots  \\
    \mychi_0 & 1 & \beta_0& 0& 0 & \cdots & 0& 0  & \cdots \\
    \mychi_1 & 1 & 1 & \beta_1 &0& \cdots & 0&0&  \cdots   \\
    \mychi_2 & 1 & 1 & 1 & \beta_2 & \cdots & 0& 0&  \cdots \\
    \mychi_3 & 1 & 1 & 1 & 1 &  \cdots &0& 0&  \cdots \\
    \vdots  & \vdots & \vdots & \vdots & \vdots & \vdots& \vdots & \vdots &  \cdots \\
    \mychi_{k} & 1 & 1 &1 &1 &\cdots &1&\beta_{k} &  \cdots \\
    \mychi_{k+1} & 1 &1 &1 & 1& \cdots & 1 & 1 & \cdots  \\
    \vdots & \vdots & \vdots & \vdots & \vdots & \vdots & \vdots& \vdots & \ddots 
  } 
  \]
  \centering
  \caption{Part of character table of $(S,*)$} \label{tabamc}
\end{table}

 First, we check that for $k \in \mathbb{Z}_+^*=\mathbb{Z}_+ \cup \{\infty\},$ $\mychi_k$  is a character. Clearly, $\mychi_\infty$ is a character. Fix $k \in \mathbb{Z}_+.$ It is easy to see that $\mychi_k(m*n)= \mychi_k(m) \mychi_k(n)$ for all $m \neq n \in \mathbb{Z}_+.$ 
 
 It remains to check that for $m \in \mathbb{Z}_+,$  $(\mychi_k(m))^2= \mychi_k(m*m),$ i.e.,
 \begin{equation} \label{M}
 (\mychi_k(m))^2=  \sum_{j=0}^m q_m(j) \mychi_k(j).
 \end{equation}
 Let $m \in \mathbb{Z}_+.$
 \begin{itemize}
 \item[Case (i)] $m \leq k$ : We have $\mychi_k(m)= 1$ and on the other hand, $\mychi_k(m*m)= \sum_{j=0}^m q_m(j) \mychi_k(j) = \sum_{j=1}^m q_m(j) =1. $ Therefore, \eqref{M} holds.
 
 \item[Case (ii)] $m=k+1$ : We have \begin{eqnarray*}
 \mychi_k(m*m)&=& \sum_{j=0}^{k+1} q_{k+1}(j) \mychi_k(j)= \sum_{j=0}^k q_{k+1}(j)+ q_{k+1}(k+1) (q_{k+1}(k+1)-1)\\ &=& (1-q_{k+1}(k+1))+(q_{k+1}(k+1))^2- q_{k+1}(k+1) \\&=& (q_{k+1}(k+1)-1)^2 = (\mychi_k (m))^2.
 \end{eqnarray*}
 Therefore, \eqref{M} holds.
 \item[Case (iii)] $m>k+1$ : We have $\mychi_k(n)=0$ for $n>k+1.$ 
  So, $\mychi_k(m*m)= \sum_{j=0}^k q_m(j)+ q_m(k+1)(q_{k+1}(k+1)-1).$
 Since by Theorem \ref{Ordered} and by proof of part (iii)(b) of Theorem \ref{Ordered}, we know that for $1 \leq j < m,$ $q_m(j)= \frac{q_m(0)}{q_j(0)}$ and $1-q_{k+1}(k+1) = q_{k+1}(0) \left( 1+ \sum_{j=1}^k \frac{1}{q_j(0)} \right),$ we get 
 $$\mychi_k(m*m)= q_m(0) \left(1+\sum_{j=1}^k 
 \frac{1}{q_j(0)} +\frac{1}{q_{k+1}(0)}(q_{k+1}(k+1)-1) \right)=0 = (\mychi_k(m))^2.$$
 \end{itemize}
 Therefore, $\mychi_k$ is a character. Also, $\mychi_k$ is real-valued and $S$ is hermitian. So, $\mychi_k$ is a symmetric character.
 
  Now, note that for $k_1 \neq k_2 \in \mathbb{Z}_+^*,$ we have $\mychi_{k_1} \neq \mychi_{k_2}.$ We will show that injective map $k \mapsto \mychi_k$ from $\mathbb{Z}_+^*$ to $\mathcal{X}_b(S)$ is onto. 
  
  To prove this take any $\chi \in \mathcal{X}_b(S)$ such that $\chi \not \equiv 1.$ Since $\chi(0)=1,$ there exists a least $s \in \mathbb{N}$ such that $\chi(s) \neq 1.$ Let $m,n \in \mathbb{Z}_+$ with $m<n.$ We have 
  $$\chi(m) \chi(n) = \chi(m*n)= \chi(n).$$
  This shows that either $\chi(m)=1$ or $\chi(n)=0.$ As a consequence, $\chi(n)=0$ for all $n>s.$
  
  Now consider the set $$A=\{j \in \mathbb{N}: \chi(t)=0 \,\text{for all} \, t \geq j \}.$$
  Then $s+1 \in A$ and so $ A \neq \phi.$ Hence, there exists a least element, say, $w$ in $A.$ 
   
   Let, if possible, $w=1.$ Then $\chi(0)=1$ and $\chi(t) =0$ for all $t \geq 1.$ So, for any $m \geq 1,$ we get $\chi(m*m)= \chi(m)\chi(m)=0.$ On the other hand, we have $\chi(m*m)=  \sum_{j=0}^m q_m(j) \chi(j)= q_m(0) >0,$ which is a contradiction. Therefore, $w>1.$ 
   
   Note that $\chi(n)=1$ for all $n < w-1.$ In fact, let,  if possible, $\chi(n_1) \neq 1$ for some $n_1<w-1.$ Therefore we get $\chi(t)=0$ for $t \geq w-1,$ which shows that $w-1 \in A.$ This is a contradiction to the fact  $w$ is the least element of $A.$ 
   
  Now, we come to $\chi(w-1)= \beta$ (say). Then, we have 
  \begin{eqnarray*}
\beta^2 = (\chi(w-1))^2 &=& \chi((w-1)*(w-1))= \sum_{j=0}^{w-1} q_{w-1}(j) \chi(j)\\& =& \sum_{j=0}^{w-2} q_{w-1}(j) + q_{w-1}(w-1) \,\beta.   
\end{eqnarray*}   
This implies that $\beta^2- q_{w-1}(w-1)\, \beta - (1-q_{w-1}(w-1))=0$  and therefore, $\beta = 1$ or $ q_{w-1}(w-1)-1=\beta_{w-2}$ where $w-2\geq 0.$ 

Let, if possible, $\beta =1.$ As $w \in A,$ $\chi(w*w)= \chi(w)\chi(w)=0;$  on the other hand, we have $\chi(w*w) = \sum_{j=0}^{w} q_w(j) \chi(j) =\sum_{j=0}^{w-1} q_w(j) \geq q_w(0) >0.$  Hence, $\beta= \beta_{w-2}.$

Therefore, the character $\chi$ is given by 
$$\chi(n) = \begin{cases} 1 &\text{if}\,\, n \leq w-2,\\ \beta_{w-2} &\text{if}\,\, n=w-1, \\ 0 & \text{if}\,\,n>w-1,
\end{cases}$$
i.e., $\chi= \mychi_{w-2}.$


 Therefore, the map $k \mapsto \mychi_k $ is onto. As an immediate consequence, $\mathcal{X}_b(S) = \widehat{S}.$ 

Further, it is clear from Table \ref{tabamc} that $\widehat{S}$ equipped with the topology of uniform convergence on compact subsets of $S$ (i.e.,  the topology of pointwise convergence in our case) is the one point compactification $\mathbb{Z}_+^*$ of $\mathbb{Z}_+.$
  \end{proof}
  
  \begin{rem} \label{rea} (i) Consider any countable infinite compact Hausdorff space $X.$ Then any  probability measure $\mu$ on $X$ has the form $\mu =  \sum_{x \in X} \mu(x) \delta_x=\sum_{x \in X} \gamma_x \delta_x$ (say), where $\gamma_x \geq 0$ and $\sum_{x \in X} \gamma_x =1.$ Trivially, $\mu$ has compact support. 
In particular, if $X=\widehat{S}$ as in Theorem \ref{St} above, $\mu = \sum_{j \in \mathbb{Z}_+} \myg_j \delta_{j}+ \myg_\infty \delta_\infty,$ where $\myg_j \geq 0\, \,\forall j \in \mathbb{Z}_+^*$ and $\sum_{j \in \mathbb{Z}_+} \myg_j+\myg_\infty =1.$  

Let $C= \left\lbrace  j \in \mathbb{Z}_+^*: \gamma_j >0\right\rbrace.$ Then the support of $\mu$ is given by  $$\overline{C}= \begin{cases}  C & \text{if}\, C \, \text{is finite or } \infty \in C, \\ C \cup \{\infty\} & \text{otherwise}.\end{cases}$$

(ii) For a non-constant complex-valued function $f$ on $\widehat{S} \cong \mathbb{Z}_+^*,$ which is eventually $1$ in the sense that $f(j)=1$ for $j\geq j_0$ for some $j_0 \in \mathbb{N}$,  we have, for  $\mu$ on $\widehat{S}$ as in (i) above, $$\int_{\widehat{S}} f d\mu = \sum_{\overset{j \in \mathbb{Z}_+}{j<j_0}}f(j) \myg_j+\left( 1 - \sum_{\overset{j \in \mathbb{Z}_+}{j<j_0}} \myg_j \right)= 1+  \sum_{\overset{j \in \mathbb{Z}_+}{j<j_0}} (f(j)-1) \myg_j.$$ 
  
(iii) For concepts, notation and details related to our next theorem,  we refer to \cite[ Chapter 2, particularly, Section 2.4]{Bloom}, \cite[Section 6]{Jewett} and \cite{Vrem1}. For the sake of convenience, we state Proposition 2.4.2 of \cite{Bloom} in the form we need as follows.
     
    {\it `` $\widehat{S}$ can be made into a hypergroup (with respect to pointwise multiplication) if for $\chi, \chi' \in \widehat{S},$ there exists a (regular) probability measure $\mu_{\chi,\chi'}=\sum_{j \in \mathbb{Z}_+} \mu_{\chi, \chi'}(j) \delta_{\chi_j}+ \mu_{\chi,\chi'}(\infty) \delta_{\chi_\infty} = \sum_{j \in \mathbb{Z}_+} \myg_j^{\chi, \chi'} \delta_{\chi_j}+ \myg_\infty^{\chi,\chi'} \delta_{\chi_\infty}$ (say) which satisfies \begin{itemize}
    \item[(a)] for $v \in \mathbb{Z}_+, $ $$ \chi(v) \chi'(v) = \sum_{j \in \mathbb{Z}_+} \myg_j^{\chi, \chi'} \mychi_j(v)+\myg_\infty^{\chi,\chi'},$$ 
    \item[(b)] the map $(\chi, \chi') \mapsto \spt(\mu_{\chi,\chi'})$ is continuous on $\mathbb{Z}_+^* \times \mathbb{Z}_+^* \rightarrow$ the space $\mathcal{C}(\mathbb{Z}_+^*)$ of compact subsets of $\mathbb{Z}_+^*$ with Michael topology,
    \item[(c)] $\mychi_\infty \in \spt(\mu_{\chi, \chi'})$ if and only if $\chi=\chi'.$'' 
    \end{itemize}}
    In this case, we may set $\delta_\chi*\delta_{\chi'}= \mu_{\chi,\chi'},$ take involution as identity and make $\widehat{S}$ a hermitian hypergroup, in fact.
     
\end{rem}

\begin{thm} \label{hp} Let $(S,<, \cdot) \cong (\mathbb{Z}_+, <, \text{max})$ and other symbols  satisfy the conditions (ii)-(iii) of Theorem 3.2 and let $(S,*)$ be the corresponding deformed hypergroup with the Haar measure $\lambda$. The dual space $\widehat{S}$ of $(S,*)$ becomes a countable compact hermitian  hypergroup with respect to pointwise multiplication. More precisely, the convolution `$*$' on $\widehat{S}$ is given by $$ \delta_{\mychi_m}*\delta_{\mychi_n} = \begin{cases} \delta_{\mychi_{\emph{min}\{m,n\}}} & \text{for\,}\,m, n \in \mathbb{Z}_+\,\,\text{with}\,\,  m \neq n\,\, \text{or}\,\,m=n=\infty, \\ \sum_{j \in \mathbb{Z}_+} \myg_j^m \delta_{\mychi_j}& \,\,\,\,\,\,\text{otherwise}, \end{cases}$$
where, $\myg_j^m= 0$ for $j<m,$ $\myg_m^m= 1+\beta_m \geq 0,\,$ and for $p \geq 1,$
 $ \myg_{m+p}^m= \prod_{j=m}^{m+p-1} \frac{-\beta_j}{1-\beta_{j+1}}>0,$ and, $\beta_j$'s are as in Theorem \ref{St}.
Further, we also have $ \mathcal{X}_b(\widehat{S})=\widehat{\widehat{S}} \cong (S,*).$ 
\end{thm}
\begin{proof}
We will freely use Remark \ref{rea}.

It is clear from Table \ref{tabamc} that $\delta_{\chi_m}*\delta_{\chi_n} := \delta_{\mychi_{\text{min}\{m,n\}}} $ for $m, n \in \mathbb{Z}_+^*$ with $m \neq n$ or $m=n= \infty.$

We now look for a suitable probability measure $\mu_{m,m},$ in short, $\mu_m,$ for $m \in \mathbb{Z}_+.$ Fix $m \in \mathbb{Z}_+.$ By  Remark \eqref{rea}(i) above, we may write $\mu_m =\sum_{j \in \mathbb{Z}_+}\myg_j^m \delta_j+ \myg_\infty^m \delta_\infty,$ where $\myg_j^m \geq 0$ for $j \in \mathbb{Z}_+^*$ and $\sum_{j \in \mathbb{Z}_+} \myg_j^m + \myg_\infty^m =1.$ Now, for any $t \in \mathbb{N},$ consider the function $e_t$ on $\widehat{S}$ defined by $e_t(\mychi_n)= \mychi_n(t)$ for $n \in \mathbb{Z}_+^*.$ We note from Table \ref{tabamc} and Remark \ref{rea}(ii) that the condition in  Remark \ref{rea} (iii)(a) takes the form
$$\mychi_m(t)^2-1 = \sum_{\underset{n<t}{n \in \mathbb{Z}_+}} (\mychi_n(t)-1) \myg_n^m \,\,\,\,\,\, \text{for all}\, t \in \mathbb{N}.$$ 
We may express it as:
 \begin{equation} \label{z}
\left(\begin{array}{cccc}  
1-\beta_0 & 0 &0 & \cdots  \\ 1& 1-\beta_1 & 0 & \cdots\\ 1 &1 & 1-\beta_2 & \cdots \\ \vdots& \vdots &\vdots & \ddots 
  \end{array} \right) \left(\begin{array}{c} \myg_0^m \\ \myg_1^m \\ \myg_2^m \\ \vdots   \end{array} \right) = \left(\begin{array}{c} 1-\mychi_m(1)^2 \\ 1-\mychi_m(2)^2 \\1-\mychi_m(3)^2 \\ \vdots  \end{array} \right). \end{equation}
  We consider any principal truncated submatrix of the infinite matrix on the L.H.S. of \eqref{z} above. It is a lower triangular matrix and its diagonal entries $1-\beta_j$'s are all non-zero, and, therefore, it is invertible.  
 Now, for $j \in \mathbb{N}$ with $j\leq m,$ $1-\mychi_m(j)^2=0.$ So we get $\myg^m_j=0$ for $j<m$ . So the system \eqref{z} can be replaced by the system:
 \begin{equation} \label{z1}
\left(\begin{array}{cccc}  
1-\beta_m & 0 &0 & \cdots  \\ 1& 1-\beta_{m+1} & 0 & \cdots\\ 1 &1 & 1-\beta_{m+2} & \cdots \\ \vdots& \vdots &\vdots & \ddots 
  \end{array} \right) \left(\begin{array}{c} \myg_m^m \\ \myg_{m+1}^m \\ \myg_{m+2}^m \\ \vdots   \end{array} \right) = \left(\begin{array}{c} 1-\beta_m^2 \\ 1 \\1 \\ \vdots  \end{array} \right) \end{equation} 
 and $\myg_j^m=0$ for $j<m,$ if any.
  
  The first two rows of \eqref{z1} readily give, 
 
 $$\myg_m^m= 1+\beta_m  \geq 0,\,\, \text{and}\,\, \myg_{m+1}^m= \frac{-\beta_m}{1-\beta_{m+1}}.$$ 
 
 We note that by Theorem \ref{St}, $\myg_m^m>0$ if and only if $\delta_{m+1}*\delta_{m+1}(m+1)>0.$
 
 Consecutive pairs of rows then give that, for $p \geq 2,$ 
 $$\myg_{m+p}^m = \frac{-\beta_{m+p-1}}{1-\beta_{m+p}} \myg_{m+p-1}^m .$$
 
 So, for $p \geq 1,$
 $$\myg_{m+p}^m = \prod_{j=m}^{m+p-1} \frac{-\beta_j}{1-\beta_{j+1}}>0.$$


Let $s_p= \sum_{j=0}^p \myg_{m+j}^m$  for $p \geq 1.$ First note that $$s_1= \myg_m^m+\myg_{m+1}^m= 1+ \beta_m-\frac{\beta_m}{1-\beta_{m+1}}= 1+ \left(\frac{-\beta_m}{1-\beta_{m+1}} \right) \beta_{m+1}= 1+\myg_{m+1}^m \beta_{m+1}.$$
Assume that $s_p= 1+\myg_{m+p}^m \beta_{m+p}.$ Then 
\begin{eqnarray*}
s_{p+1} &=& s_p+ \myg_{m+p+1}^m=1+ \myg_{m+p}^m\, \beta_{m+p}+ \myg_{m+p+1}^m \\ &=& 1+ \myg_{m+p}^m \beta_{m+p}- \myg_{m+p}^m \frac{\beta_{m+p}}{1-\beta_{m+p+1}}\\ &= & 1+ \myg_{m+p}^m\left( \frac{-\beta_{m+p}}{1-\beta_{m+p+1}} \right) \beta_{m+p+1} = 1+\myg_{m+p+1}^m \beta_{m+p+1}.
\end{eqnarray*} 
Therefore, by mathematical induction we get $s_p= 1+ \myg_{m+p}^m \beta_{m+p}$ for $p \geq 1.$

Let $p \geq 1.$  Then $$\myg_{m+p}^m \,\beta_{m+p}= \beta_m \prod_{j=1}^{p} \frac{-\beta_{m+j}}{1-\beta_{m+j}},$$


which, by using the fact that $-1 \leq  \beta_k <0$ for all $k \in \mathbb{Z}_+,$ gives $$0<-\myg_{m+p}^m \,\beta_{m+p}\leq \left( \frac{1}{2} \right)^{p}.$$
Therefore, $\sum_{j \in \mathbb{Z}_+} \myg_j^m = \lim_{p \rightarrow \infty} s_p = 1.$ So, $\myg_\infty^m= 1- \sum_{j \in \mathbb{Z}_+} \myg_j^m = 0.$
 
  Hence, we may set $\mu_m = \delta_{\mychi_m}*\delta_{\mychi_m} = \sum_{j \in \mathbb{Z}_+}  \myg_j^m\, \delta_{\mychi_j}.$ 
 
 The continuity of the map $(\mychi_m, \mychi_n) \mapsto \spt(\delta_{\mychi_m}*\delta_{\mychi_n})$  is easy to  see by  noting that for $m \neq n$ or $m=n= \infty,\,\spt(\delta_{\mychi_m}*\delta_{\mychi_n}) $ is singleton $\{\mychi_{\text{min}\{m,n\}}\}$ and for $m \in \mathbb{Z}_+,$ $\spt(\delta_{\mychi_m}*\delta_{\mychi_m}) = \left\lbrace \mychi_k \in \widehat{S}: k \geq m \right\rbrace \cup \{\mychi_\infty\}$ or $\left\lbrace \mychi_k \in \widehat{S}: k > m \right\rbrace \cup \{\mychi_\infty\}$ according as $\delta_{m+1}*\delta_{m+1}(m+1)>0$ or $\delta_{m+1}*\delta_{m+1}(m+1)=0.$  
 
It can be readily checked that $\mychi_\infty$ is the identity element. We see that $\mychi_\infty \in \spt(\delta_{\mychi_m}*\delta_{\mychi_n}) $ if and only if $m=n.$ In view of Remark \eqref{rea}(iii),  $\widehat{S}$ is a compact hermitian hypergroup.


For the rest of the proof, it is convenient to index the elements of  $\widehat{S}$ by $\mathbb{Z}_+^*$ under the identification $k \mapsto \mychi_k.$

Now, we calculate the dual space $\mathcal{X}_b(\widehat{S}),$ the space of continuous characters on $\widehat{S}.$ We know that $S \subset \widehat{\widehat{S}}$ via $n \mapsto \myxi_n,$ defined as $\myxi_n(\mychi_k)= \myxi_n(k):= \mychi_k(n)$ for all $k \in \mathbb{Z}_+^*.$ We can figure them out from Table \ref{tabamc}; to elaborate, $\myxi_0 \equiv 1,\, \myxi_1(0)=\beta_0,\, \myxi_1(k)=1$ for $k \geq 1,$ and, for $n\geq 2,$  
$$\myxi_n(k)= \begin{cases} 0 & \text{if}\,\,k<n-1, \\ \beta_{n-1} & \text{if}\,\,k=n-1, \\ 1 &\text{if}\,\, k \geq n. \end{cases}$$

Next, let $ \xi \in \mathcal{X}_b(\widehat{S})$ such that $\xi \not\equiv 1.$ Consider the set 
$$B= \left\lbrace s \in \mathbb{Z}_+: \xi(s) \neq 1 \right\rbrace.$$ 

Then $B \neq \phi.$ Take any $s \in B,$ so $\xi(s)\neq 1.$ Note that for $m \in \mathbb{Z}_+, n \in \mathbb{Z}_+^*$  with $m<n,$ we have $\xi(m)= \xi(m)\xi(n),$ and therefore,  either $\xi(m)=0$ or $\xi(n)=1.$ In particular, we get that for all  $t \in \mathbb{Z}_+ $ with $t <s,$ $\xi(t)=0$  and hence $t \in B.$ 
This implies that $\mathcal{L}_s \cup \{s\} \subset B.$ So either $B= \mathbb{Z}_+$ or $B$ has a maximum element. 

Let, if possible, $B= \mathbb{Z}_+.$ Then, for each $n \in \mathbb{N},$ $\xi=0$ on $\mathcal{L}_n.$ Since this holds for every $n \in \mathbb{N}$ we get $ \xi \equiv 0 $ on $\mathbb{Z}_+.$ Since $\xi(\infty)=1,$ we get that $\xi$ is not continuous, a contradiction. Hence, $B$ has a maximum element, say $w.$  Therefore, it follows from the definition of $w$ that $\xi(k)= 1$ for all $k>w.$ Also, since $\xi(w) \neq 1$ we get $ \xi(k)=0$ for $k <w.$ 

Now, we calculate the value of $\xi(w) = \beta$ (say). We have $$\beta^2= \xi(w)^2= \xi(w*w)= \myg_w^w \beta+\sum_{\underset{j>w}{j \in  \mathbb{Z}_+}} \myg_j^w.  $$ 
This implies that $\beta^2- \myg_w^w \beta -(1-\myg_w^w)=0$ and therefore, $\beta=1$ or $\myg_w^w-1.$ But we know that $\beta = \xi(w) \neq 1$ and hence $\xi(w)= \myg_w^w-1=\beta_w.$
Now, set $n=w+1 \geq 1,$ then $\xi= \xi_n.$

 Hence $(S,*) \cong \mathcal{X}_b(\widehat{S}).$ \end{proof}  
 
 \begin{rem} (i)(a) The last part of Theorem \ref{hp} can be proved directly with the help of \cite[Proposition 2.4.18]{Bloom}. Our proof is easy and self contained.
 
 (b) $\gamma$'s in Theorem \ref{hp} are related via $\myg_j^m= \frac{\myg_j^0}{\myg_m^0}$ for $j>m>0.$
 
 (ii)(a) Any hypergroup $(\widehat{S},*)$ as in Theorem \ref{hp} above may be thought of as hypergroup deformation of the compact linearly ordered semigroup $(\mathbb{Z}_+^*, <, \text{min})$  with the order topology (cf. \cite{Zuck}, \cite{Ross}, \cite{Hofmann}) on the diagonal of $\mathbb{Z}_+ \times \mathbb{Z}_+.$
 
 (b) The class of  hypergroups $\widehat{S}$ as in Theorem \ref{hp} above lies between the class of countable compact hypergroups $H_a$ (see Section 2 above) of \cite{Ramirez} and  the class of compact almost discrete hypergroups of \cite{Voit}.
 
 (iii) Let $(S,<,\text{max})$ be a finite ``max" semigroup of cardinality $n.$ Then it can be identified with $(\mathbb{Z}_n, <, \text{max}).$  Suitably truncated table of Table \ref{tabamc} and truncated submatrices in \eqref{z} and \eqref{z1} can be used to formulate and prove suitable analogues of Theorem \ref{St} and Theorem \ref{hp}. 
 \end{rem}

  \section{semiconvos or hypergroups arising from deformations of commutative discrete semigroups}
We consider  a commutative discrete semigroup $(S,\cdot)$  with identity $e$ such that $E_0(S)$ is not empty. We try to make $(S,\cdot)$ into a commutative discrete semiconvo or discrete hypergroup $(S,*)$ by deforming the product on $\mathcal{D}_{E_0(S)}:=\{(m,m): m \in E_0(S)\},$ the diagonal of $E_0(S),$ or the idempotent diagonal of $S,$ say.
 
 For $n \in E(S),$ let $q_n$ be a probability measure on $S$ with finite support $Q_n$ containing $e.$ We express $q_n= \sum_{j \in Q_n} q_n(j) \delta_j$ with $q_n(j)>0$ for $j \in Q_n$ and $\sum_{j \in Q_n} q_n(j)=1.$ We look for necessary and sufficient conditions on $S$ and $ \{ q_n: n \in E_0(S)\}$ such that $(S, *)$ with `$*$' defined below, is a commutative discrete semiconvo:
 \begin{eqnarray*}
 \delta_m*\delta_n =  \delta_n*\delta_m & = &\delta_{mn} \,\,\,\,\mbox{for}\,\, (m,n) \in S\times S \backslash \mathcal{D}_{E_0(S)},   \\
 \delta_n*\delta_n &= &q_n \,\,\,\,\mbox{for}\,\, n \in E_0(S). 
 \end{eqnarray*}

It is convenient to consider various relationships among the properties of $(S,\cdot)$ and of $(S,*)$ presented in the following theorems before coming to our main Theorem \ref{main}.
\begin{thm} \label{equivalence} Let $(S,\cdot)$ be a commutative discrete semigroup with identity $e.$ Suppose $(S,*)$ is a commutative discrete semiconvo, where `$*$' and other related notation and concepts are as above. Then the following are equivalent. 
\begin{itemize}
\item[(i)]  For $n \in E_0(S),$\, $Q_n \subset E(S).$
\item[(ii)]  $(E(S), \cdot)$ is a max-min type semigroup.
\item[(iii)]  $S$ is inverse-free.
\item[(iv)] $G_1(S)=\{e\},$ i.e., $S$ is action-free.
\item[(v)]For $n \in E_0(S)$ and $m \in \widetilde{S}\, (= S \backslash E(S))$ we have $n \neq nm.$
\end{itemize}
\end{thm}
\begin{proof}
We prove the following implications: (i) $\Rightarrow$ (iv) $\Rightarrow$ (iii) $\Rightarrow$ (v) $\Rightarrow$ (i) and (iii) $\Leftrightarrow $ (ii).
\begin{itemize}
\item[(i) $\Rightarrow $(iv)]. Assume $ Q_n \subset E(S)$ for $n \in E_0(S).$ Consider any such $n.$ Let, if possible, $G_1(S) \neq \{e\}$ and  choose any $m (\neq e) \in G_1(S)$ so that $mn=n \neq m.$ As $(S,*)$ is a semiconvo, the associativity of `$*$' demands that $(\delta_m*\delta_n)*\delta_n=\delta_m*(\delta_n*\delta_n).$ L.H.S. is equal to $\delta_{mn}*\delta_n = \delta_n*\delta_n=q_n$ and R.H.S. is equal to $\delta_m*q_n.$ Since both sides are equal, we get $q_n= q_n*\delta_m.$ Now, $m \in \spt(q_n*\delta_m).$ So $m \in Q_n,$ which is a contradiction. Hence, $G_1(S)=\{e\},$ i.e., $S$ is action-free.  
\item[(iv) $\Rightarrow$ (iii)]. Assume that $S$ is not inverse-free. Then $G(S)$ does not act as the identity on $E_0(S).$ So there exists $k \in E_0(S)$ such that $km \neq k$ for some $m (\neq e) \in G(S).$ Note that $k \neq e$ and also $\delta_k \neq q_k$ since $e \in \spt(q_k).$
Now, there exists $n (\neq e)  \in G(S)$  such that $mn=e.$ Note that $kn \neq n$ because if $kn=n$ then $k=e$ by multiplying by $m$ on both sides, which is not true. Similarly, $km \neq m.$
Consider the three different cases:\\ (A) $km \neq kn,$\,\, (B) $km =kn \notin E_0(S),$ and \,\,(C) $km=kn \in E_0(S).$ 

 Now, because $(S,*)$ is a semiconvo, the associativity of $`*$' demands that
\eqspecialnum \begin{equation}\label{3} 
(\delta_{km}*\delta_{kn})*\delta_k = \delta_{km}*(\delta_{kn}*\delta_k).
\end{equation}

Case (A). When $km \neq kn.$ First, note that $kn \neq k$ because if $kn=k$ then by multiplying both sides by $m,$ we get $k=km$ so that $kn=k=km,$ which can not happen. The L.H.S. of \eqref{3} is equal to $\delta_{k^2mn}* \delta_k= q_k$ because $k \in E_0(S)$ and R.H.S. of \eqref{3} $= \delta_{km}*\delta_{k^2n}= \delta_{km}*\delta_{kn} = \delta_{k^2mn}= \delta_k$ because $km \neq kn$ and $k \in E_0(S).$ Since $\delta_k \neq q_k,$ \eqref{3} does not hold.  

Case (B). When $km =kn \notin E_0(S).$ Then L.H.S. of \eqref{3} is equal to $\delta_{k^2mn}*\delta_k = \delta_k*\delta_k=q_k$ because $k \in E_0(S).$ And R.H.S. of \eqref{3} is equal to  $ \delta_{km}*\delta_{kn} = \delta_{k^2mn}=\delta_k$ because $kn \notin E_0(S)$ and $k \in E_0(S).$  Therefore, \eqref{3} does not hold.

Case (C). When $km = kn \in E_0(S).$ As $(S,*)$ is a semiconvo, the associativity of `$*$' demands that
\eqspecialnum \begin{equation}\label{4} 
(\delta_{k}*\delta_{kn})*\delta_{n} = \delta_{k}*(\delta_{km}*\delta_n).
\end{equation}
Now, L.H.S.= $\delta_{k^2n}*\delta_n= \delta_{kn}*\delta_n= \delta_{kn^2}$ and R.H.S.= $\delta_{k}*\delta_{kmn}=\delta_k*\delta_k =q_k.$ Since $e\in \spt(q_k)= \spt(\text{R.H.S.})=\spt(\text{L.H.S.}) = \{kn^2\},$ we get $kn^2=e,$ so that $k \in G(S)$ and $k \in G(S) \cap E(S)= \{e\},$ which is a contradiction. Therefore, \eqref{4} does not hold.  
 
 Therefore, in all three cases the associativity condition  in Definition \ref{semiconvo} does not hold so that $(S,*)$ is not a semiconvo, which is a contradiction to the assumption. Hence, $S$ is inverse-free.
 \item[(iii) $\Rightarrow$ (v)]. Let, if possible, $n=nm.$ The associativity of `$*$' demands that for $m,n \in S, \,(\delta_n*\delta_n)*\delta_m = \delta_n*(\delta_n*\delta_m).$ Now, as $n \in E_0(S)$ and $m \in \widetilde{S},$ we get $q_n*\delta_m= \delta_n*\delta_{nm}= q_{n}.$  Therefore, $\sum_{j \in Q_n} q_n(j) (\delta_j*\delta_m)= \sum_{k \in Q_n} q_n(k) \delta_k.$ But as $m \in \widetilde{S}, $ we get $$ q_n(e) \delta_m + \sum_{e \neq j \in Q_n} q_n(j) \delta_{jm}= q_n(e) \delta_e+ \sum _{e \neq k \in Q_n} q_n(k) \delta_k.$$
This shows that there exists a $j_e (\neq e) \in Q_n $ such that $j_em=e$ which implies that $j_e=m=e $ because $S$ is inverse-free. This is a contradiction.
 Therefore, $n \neq nm.$ 
 \item[(v)$\Rightarrow$ (i)]. Let, if possible, there exist $ m \notin E(S),$ which is in $Q_n.$ Since $n \neq m,$ by the associativity of `$*$', we have  $q_n*\delta_m = \delta_n*\delta_{nm}.$ But $n \neq nm$, so we have that $\delta_n*\delta_{nm}=\delta_{n^2m}= \delta_{nm}.$ Therefore,  
$ \sum_{j \in Q_n} q_n(j)(\delta_j*\delta_m)= \delta_{nm}.$  By the definition of `$*$', we get
$$ q_n(e) \delta_m + \sum _{e \neq j \neq m,\, j \in Q_n } q_n(j) \delta_{jm}+ q_n(m) \delta_{m^2}= \delta_{nm}. $$ So $m=nm$ and $m^2=nm.$ This shows that $m=m^2$ and hence $m \in E(S),$ which is a contradiction. Therefore $Q_n \subset E(S).$
\item[(iii) $\Leftrightarrow$ (ii)].  Suppose (iii) holds. We know that $(E(S), \cdot)$ is a subsemigroup of $(S,\cdot).$ Let, if possible, there exist $m,n \in E(S)$ with $m \neq mn \neq n.$ Then $e \neq m \neq n \neq e.$ So, by Proposition \ref{col E0}(iv), $mn \in E_0(S).$ Now, the associativity of `$*$' demands  $$(\delta_m*\delta_n)*\delta_{mn}= \delta_m*(\delta_n*\delta_{mn}).$$ 
We see that L.H.S. = $\delta_{mn}*\delta_{mn}= q_{mn}.$ Further, R.H.S. = $\delta_m*\delta_{nmn}= \delta_m*\delta_{mn}$ using the fact that $n \in E(S).$ As $m \neq mn$ we get R.H.S.= $\delta_{m^2n},$ which in turn $= \delta_{mn}$ because $m \in E(S).$ Thus $q_{mn}= \delta_{mn}.$ Now, $e \in Q_{mn}$ so $mn=e;$ this is not true as it shows that $m,n \in G(S),$ which is a contradiction. Hence $E(S)$ is a max-min type semigroup, i.e., (ii) holds.

The reverse implication follows from Proposition \ref{semi}(ii).\end{itemize} \end{proof}

\begin{thm} \label{Supportive} Let $(S,\cdot)$ be a commutative discrete semigroup with identity $e.$ Suppose $(S,*)$ is a commutative discrete semiconvo, where `$*$' and other related notation and concepts are as above. Then we have the following:
\begin{itemize}
\item[(i)] For $m,n \in E_0(S)$ with $m \neq n \neq nm,$ we have $m \notin Q_n$ and $Q_n \cdot m= \{nm\}.$
\item[(ii)]  Under any (hence all) conditions (i)-(v) of Theorem \ref{equivalence}, we have the following.
\begin{itemize}
\item[(a)] $(E(S),*)$ is a hermitian (hence commutative) discrete hypergroup.
\item[(b)] $E(S)$ is finite or $E(S)$ is isomorphic to $(\mathbb{Z}_+,<,\text{max}),$ where the order on $E(S)$ is defined by $m<n$ if $mn=n \neq m.$
\item[(c)] If $\#E(S) > 2,$ then for $ e \neq m <n $ in $E(S),$ we have the following.
 \begin{itemize}
 \item[($\alpha$)] $q_n(e)=q_n(m) q_m(e)$ and 
 \item[($\beta$)] $q_n(e) \left( 1+ \sum_{e \neq k \in \mathcal{L}_n} \frac{1}{q_k(e)} \right) \leq 1,$
 \end{itemize}
 where for $n \in E(S),$ $\mathcal{L}_n := \{j \in E(S): j<n\} .$
\item[(d)] $\widetilde{S}$ is an ideal of $(S, \cdot).$ 
\item[(e)] For $n \in E_0(S), m \in \widetilde{S}$ we get $Q_n\cdot m=\{nm\}.$ 
\end{itemize} 
\end{itemize}
\end{thm}
\begin{proof}
(i). Let $m,n \in E_0(S)$ with $m \neq n \neq nm.$ The associativity of `$*$' demands that $(\delta_n*\delta_n)* \delta_m = \delta_n *(\delta_n *\delta_m).$ Now L.H.S. $= q_n*\delta_m.$ Further, R.H.S. $= \delta_n* \delta_{nm}$ which in turn $= \delta_{nm}$ using the fact that $n \neq nm$ and $n \in E_0(S).$ Thus $q_n*\delta_m= \delta_{nm}.$  Let, if possible $m \in Q_n.$ Then L.H.S. = $\sum_{j \in Q_n} q_n(j) (\delta_j*\delta_m)= q_n(e) \delta_m+ q_n(m) q_m+ \sum_{\underset{e \neq j \neq m}{j \in Q_n}} q_n(j) \delta_{jm}.$  Since $e \in Q_m$ and $e \neq m,$ the measure on L.H.S. is non-Dirac but R.H.S. is a Dirac measure, which is a contradiction. Therefore, $m \notin Q_n$ and hence, by $q_n*\delta_m= \delta_{nm},$ we get $Q_n \cdot m = \{nm\}.$

(ii)(a).  We know that $(E(S), \cdot) $ is a subsemigroup of $(S, \cdot)$. Under the assumption, (i) in Theorem \ref{equivalence} is satisfied. Therefore, `$*$' induces a product on $E(S),$ again denoted by `$*$', i.e., $*= *|_{E(S)} : E(S) \times E(S) \rightarrow M_{F,p}(E(S)).$ Since we assume that $e \in Q_n,$ using (iii) of Theorem \ref{equivalence} it is easy to see that for all $m,n \in E(S), \,e \in \spt(\delta_m*\delta_n)$ if and only if $m=n.$ The associativity follows as $(S,*)$ is a discrete semiconvo. Hence, `$*$' induces a hermitian discrete hypergroup structure on $E(S).$

(b). Under the assumption that $E(S)$ is a max-min type semigroup, by Item \ref{lsg}(ii), it can be given a linear order by $m<n$ if $mn=n \neq m$ and made a "max" semigroup with identity $e.$ Therefore, it follows from Theorem \ref{Ordered} that $E(S)$ is finite or $E(S)$ is isomorphic to $(\mathbb{Z}_+,<,\text{max}).$

(c). This follows by (b) above as $(E(S), \cdot)$ is a ``max" semigroup and $(E(S),*)$ is a hermitian discrete hypergroup.  

(d). Let $k,n \in \widetilde{S};$ then by the associativity of `$*$' we get $(\delta_{nk}*\delta_n)*\delta_k= \delta_{nk}*(\delta_n*\delta_k)$ which is equivalent to $\delta_{(nk)^2}= \delta_{nk}*\delta_{nk}$ as $n,k \in \widetilde{S}.$ Under the assumption $S$ is inverse-free, it follows that $nk \neq e.$ This forces $nk \notin E_0(S).$   Therefore, $\widetilde{S}$ is a subsemigroup of $(S, \cdot)$. Now, to show that $ \widetilde{S}$ is an ideal it is enough to show that for $k \in E_0(S)$ and $n \in \widetilde{S}, \,nk \in \widetilde{S}.$ This can be proved in a manner similar to the proof that $\tilde{S}$ is a subsemigroup by using $k \neq nk$ and the associativity condition $(\delta_{nk}*\delta_k)*\delta_n= \delta_{nk}*(\delta_k*\delta_n).$

(e).  Let $n \in E_0(S), m \in \widetilde{S}.$ The associativity of `$*$' demands that $(\delta_n*\delta_n)*\delta_m = \delta_n*(\delta_n*\delta_m).$ By the definition of `$*$', we get $q_n*\delta_m=\delta_n*\delta_{nm}.$ Since, by (d) above, $n \neq nm,$ we get $q_n*\delta_m = \delta_{nm},$ which, using (i), shows that $Q_n \cdot m = \{nm\}.$  
\end{proof}

Now, we state our main theorem.
 \begin{thm} \label{main} Let $(S,\cdot)$ be a  commutative discrete semigroup with identity $e$ such that $S$ is action-free. Let `$*$' and other related notation and concepts be as above. Then $(S, *)$ is a commutative discrete semiconvo if and only if the following conditions hold.
 \begin{itemize}
 \item[(i)]$E(S)$ is finite or $E(S)$ is isomorphic to $(\mathbb{Z}_+,<,\text{max}),$ where the order on $E(S)$ is defined by $m<n$ if $mn=n \neq m.$
  \item[(ii)] $(\widetilde{S}, \cdot)$ is an ideal of $(S, \cdot).$  
  \item[(iii)] $Q_n \subset E(S)$ for $n \in E_0(S).$
 \item[(iv)] If $n \in E_0(S)$ and $m \in \widetilde{S}$ then $Q_n\cdot m =\{nm\}.$  
 \item[(v)] For $n \in E_0(S),$ we have $\mathcal{L}_n \subset Q_n \subset \mathcal{L}_n \cup \{n\},$  where for $n \in E(S),$ $\mathcal{L}_n := \{j \in E(S): j<n\}.$ 
 \item[(vi)]If $\#E(S) > 2,$ then for $ e \neq m <n $ in $E(S),$ we have the following:
 \begin{itemize}
 \item[($\alpha$)] $q_n(e)=q_n(m) q_m(e)$ and 
 \item[($\beta$)] $q_n(e) \left( 1+ \sum_{e \neq k \in \mathcal{L}_n} \frac{1}{q_k(e)} \right) \leq 1.$
 \end{itemize}

 \end{itemize}
 Further, under these conditions, $E(S)$ is a hermitian discrete hypergroup. Moreover, $S$ is a hermitian discrete hypergroup if and only if $S = E(S).$
 \end{thm}
 \begin{proof} We assume all the conditions and prove that $(S,*)$ is a commutative discrete semiconvo. Since $S$ has an identity, it is enough to check the associativity of `$*$' to prove that $(S,*)$ is a semiconvo. We will prove that for $m,n,k \in S$ 
 \eqspecialnum
\begin{equation}\label{2}
(\delta_m*\delta_n)*\delta_k = \delta_m*(\delta_n*\delta_k). 
\end{equation} 

  We divide the proof of \eqref{2} into various steps.
\begin{itemize}
\item[(i)] If $m,n,k \in \widetilde{S},$ then using the fact the $\widetilde{S}$ is an ideal (hence subsemigroup) of $S$ we have that the product of any two elements is in $\widetilde{S}.$ Therefore, both sides of \eqref{2} become $\delta_{mnk}.$ 
\item[(ii)] If $m,n \in \widetilde{S}$ and $k \in E(S)$ then both sides of \eqref{2} become $\delta_{mnk}$ using the fact 
$\widetilde{S}$ is an ideal of $S.$    
\item[(iii)] Suppose $m \in \widetilde{S}$ and $ n,k \in E(S).$ If $n$ or $k$ is $e$ then \eqref{2} is trivial. Therefore, we can assume that $n,k \in E_0(S).$
\begin{itemize}
\item[(a)] When $n=k \in E_0(S).$ 
The condition \eqref{2} becomes  $(\delta_m*\delta_n)*\delta_n= \delta_m*(\delta_n*\delta_n).$ Now, $\delta_m*(\delta_n*\delta_n)= \delta_m*q_n = \sum_{j \in Q_n} q_n(j) \, (\delta_m*\delta_j).$ Because for $n \in E_0(S),$ $Q_n \subset E(S)$ and $m \cdot Q_n = \{mn\},$ we get R.H.S. of \eqref{2}
\begin{eqnarray*}
 = \delta_m*(\delta_n*\delta_n)= \sum_{j \in Q_n} q_n(j) \delta_{mj}=  \sum_{j \in Q_n} q_n(j) \delta_{mn} = \left( \sum_{j \in Q_n} q_n(j) \right) \delta_{mn}  = \delta_{mn}.
\end{eqnarray*}

 Using the fact $\widetilde{S}$ is an ideal and $n^2=n,$ we get L.H.S. of \eqref{2} $=(\delta_m*\delta_n)*\delta_n= \delta_{mn}.$ Therefore, the  associativity \eqref{2}  holds. 

\item[(b)] For $n \neq k \in E_0(S),$ R.H.S. and L.H.S. of  \eqref{2} both  become $\delta_{mnk}$ using the fact the $\widetilde{S}$ is an ideal in $(S, \cdot).$ Therefore, \eqref{2} holds.
 \end{itemize} 
 \item[(iv)] We now come to cases when $n \in \widetilde{S}$ and $m,k \in E(S),$ or $k \in \widetilde{S}$ and $m,n \in E(S).$ These can be dealt with by arguments similar to those in (ii) and (iii) above.
 
 \item[(v)] When $m,n,k \in E(S),$ associativity of `$*$' can be proved in a manner similar to that for proof of Theorem \ref{Ordered}.
 \end{itemize} 
 Hence, $(S, *)$ is a semiconvo.
 
Converse part of the first part of the theorem follows from Theorem \ref{Ordered} and  Theorem \ref{Supportive}.

Since $(S,*)$ has a semiconvo structure and $(S,\cdot)$ is action-free, it  follows from Theorem \ref{Supportive} that $(E(S),*)$ has a hermitian hypergroup structure.

 Moreover, if $S$ is hermitian hypergroup then for any $ e \neq m \in S,$ $e \in \spt(\delta_m* \delta_m);$ so we get $m \in E_0(S)$ or $m \in \widetilde{S}$ with $m^2=e.$ But $\widetilde{S}$ is an ideal and $e$ in not in $\widetilde{S}.$ So the second condition is not possible, which shows that $m \in E_0(S).$ Therefore, $S \subset E(S).$ Hence $S=E(S).$     \end{proof}

 \begin{rem} \begin{itemize}
 \item[(i)] We can formulate a suitable analogue of Corollary \ref{coro} based on Theorem \ref{main} above.
  
 \item[(ii)] In view of Theorem \ref{Ordered} and Theorem \ref{main} above, identifying $E(S)$ as a finite subset of $\mathbb{Z}_+$ or $\mathbb{Z}_+$ itself, for $n>1, \, \#Q_n \geq 2.$ But $\#Q_1$ can very well be $1.$ As already remarked, this is the case for Dunkl-Ramirez hypergroup for $a= \frac{1}{2}.$ We note that some proofs become simpler, shorter or different when we take $\#Q_n > 1$ for all $n.$
 
 \item[(iii)] R. C. Vrem \cite{Vrem} defined the join of two hypergroups. Because $S= E(S) \sqcup \widetilde{S},$ at the first  glance, our semiconvo $(S,*)$ looks like the join of a hermitian discrete hypergroup $(E(S),*)$ and a discrete semigroup $(\widetilde{S}, \cdot),$ but it is something different from join. 
 \item[(iv)] We can utilize the examples in Section \ref{Basic} to illustrate various parts of  theorems above. In particular, it is a consequence of Theorem \ref{equivalence} that for Example \ref{2.3}, `$*$' does not induce a semiconvo structure on $S.$ 
 \end{itemize}
 \end{rem}
 The following semigroup $(S, \cdot)$ is an example of a semigroup $(S, \cdot)$ satisfying the conditions of Theorem \ref{main} and for which $\widetilde{S}$ is non-empty.
 \begin{exmp} \label{max-sum}For $T= \{1-\frac{1}{r+1}: r \in \mathbb{Z}_+ \},$ let $S = T\cup \mathbb{N}$ with `$\cdot$' defined as follows.
 
$$m\cdot n  = \begin{cases} \text{max}\{m,n\} & \text{if}\,\, m\,\, \text{or}\,\,n \in T, \\ m+n & \text{if}\,\, m,n \in \mathbb{N}. \end{cases}$$
 
 Then $S$ is an inverse-free commutative discrete semigroup with identity $0.$ Further, $E(S)=T$, which is isomorphic to $(\mathbb{Z}_+,<, \cdot).$ Also, $\widetilde{S}=\mathbb{N}$ is a prime ideal in $S.$
\end{exmp} 

\section*{Acknowledgment}

Vishvesh Kumar thanks the Council of Scientific and Industrial Research, India, for its senior research fellowship. He thanks his supervisors Ritumoni Sarma and N. Shravan Kumar for their support and encouragement.

A preliminary version of a part of this paper was included in the invited talk by Ajit Iqbal Singh at the conference ``The Stone-$\check{\mbox{C}}$ech compactification : Theory and Applications, at Centre for Mathematical Sciences, University of Cambridge, July 6-8 2016" in honour of Neil Hindman and Dona Strauss. She is grateful to the organizers H.G. Dales and Imre Leader for the kind invitation, hospitality and travel support. She thanks them, Dona Strauss and Neil Hindman and other participants for useful discussion. She expresses her thanks to the Indian National Science Academy for the position of INSA Emeritus Scientist and travel support.

\end{document}